\documentclass[superscriptaddress,eqsecnum ]{revtex4-1}
\usepackage{graphicx}
\usepackage{amsmath}
\usepackage{amssymb}
\usepackage{amscd}
\usepackage{amsthm}
\usepackage{hyperref}
\usepackage{color}
\usepackage{enumerate}
\usepackage{caption}
\usepackage{subcaption}
\usepackage{adjustbox}
\usepackage[top=1in, bottom=1in, left=1in, right=1in]{geometry}
\usepackage{natbib}

\newtheorem{theorem}{Theorem}
\numberwithin{theorem}{section}
\newtheorem{lemma}[theorem]{Lemma}
\newtheorem{corollary}[theorem]{Corollary}
\newtheorem{proposition}[theorem]{Proposition}
\newtheorem*{remark}{Remark}

\begin{document}

\title{Canard-like phenomena in piecewise-smooth Van der Pol systems}
\date{\today}

\author{Andrew Roberts}
\affiliation{Department of Mathematics, University of North Carolina}
\author{Paul Glendinning}
\affiliation{Department of Mathematics, University of Manchester}

\begin{abstract}
We show that a nonlinear, piecewise-smooth, planar dynamical system can exhibit canard phenomena.  Canard phenomena in nonlinear, piecewise-smooth systems can be qualitatively more similar to the phenomena in smooth systems than piecewise-linear systems, since the nonlinearity allows for canards to transition from small cycles to canards ``with heads."  The canards are born of a bifurcation that occurs as the slow-nullcline coincides with the splitting manifold.  However, there are conditions under which this bifurcation leads to a phenomenon called super-explosion, the instantaneous transition from a globally attracting periodic orbit to relaxations oscillations.  Also, we demonstrate that the bifurcation---whether leading to canards or super-explosion---can be subcritical.  
\end{abstract}

\maketitle

{\bf Fast/slow systems arise in many physical applications.  In these systems, special trajectories that remain near repelling slow manifolds for $\mathcal{O}(1)$ time, called canards, can arise in exponentially small parameter ranges.  They are born of a Hopf bifurcation in systems with an `S'-shaped fast nullcline, and always pass near one of the fold points of the nullcline.  The nature of the canard cycle is determined by the local dynamics near the fold as well as the global dynamics of the full system.  In a piecewise-smooth system the fast nullcline may not be the graph of a differentiable function, so the local extrema might be corners.  We call the nullcline `2'-shaped or `Z'-shaped in the event of one or two corners, respectively.  At a corner the bifurcation and local dynamics can be completely different from what occurs at a smooth fold.  We find conditions under which planar piecewise-smooth fast/slow systems exhibit canard cycles due to a non-smooth `Hopf-like' bifurcation.  If the conditions are not met, the bifurcation causes the system to jump instantaneously from having a stable equilibrium to relaxation oscillations.  This bifurcation is called a super-explosion, and we also find conditions under which the system undergoes a subcritical super-explosion---that is, a stable equilibrium and relaxation oscillation exist simultaneously.  Finally, we apply the results in a physical setting, examining a variant of Stommel's ocean circulation model. }

\section{Introduction}
In dynamical systems, a canard is a trajectory of a fast/slow system that remains near a repelling slow manifold for $\mathcal{O}(1)$ time.  In smooth, planar systems, a Hopf bifurcation may occur when the slow nullcline transversely intersects the fast nullcline, also called the {\it critical manifold}, near a fold (i.e., local extremum).  If the fast nullcline is `S'-shaped, the Hopf cycles will grow to become relaxation oscillations.  The transition from Hopf cycles to relaxation oscillations happens in an exponentially small parameter range, and the phenomenon is called a {\it canard explosion}.

The theory used to analyze fast/slow systems is called geometric singular perturbation theory (GSPT).  For an introduction to GSPT, we direct the reader to the paper by Jones \cite{gsp}.  The basics of GSPT break down at fold points of the critical manifold because the fast and slow dynamics become tangent (i.e., there is no separation of time scales locally).  If the fast and slow nullclines do not intersect at a fold, then the fold point is a singularity of the reduced problem.  At a Hopf bifurcation, however, the fold point becomes a removable singularity of the reduced problem and is called a {\it canard point} \cite{ksGSP}.  In the singular limit, a canard point allows trajectories to cross from a stable branch of the critical manifold to an unstable branch of the critical manifold (or vice versa).  If the system is no longer smooth, but instead only piecewise-smooth, the analog of a canard point may no longer be a removable singularity.  


We will consider a nonlinear, piecewise-smooth, Li\'{e}nard system of the form
\begin{equation}
	\label{genIntro}
	\begin{array}{l}
		\dot{x} = -y + F(x) \\
		\dot{y} = \epsilon ( x - \lambda)
	\end{array}
\end{equation}
where $$ F(x) = \left\{ \begin{array}{ll}
	g(x) & x\leq 0 \\
	h(x) & x \geq 0
	\end{array} \right. $$ 
with $g,h \in C^k, \ k \geq 1$, $g(0) = h(0) =0$, $g'(0) < 0$ and $h'(0) > 0$, and we assume that $h$ has a maximum at $x_M > 0$.  The critical manifold 
$$ N_0 = \{ y = F(x) \} $$ is `2'-shaped with a smooth fold at $x_M$ and a corner (i.e., nonsmooth fold) at $x = 0$.  An example of such a critical manifold is shown in Figure \ref{fig:2m0}.   The line $x=0$ where $F(x)$ is not differentiable is called the {\it splitting line}. 

When $\lambda = 0$ in \eqref{genIntro}, the slow nullcline passes through the corner of the critical manifold creating the analog of a canard point.  There is still a `Hopf-like' bifurcation at $\lambda=0$, with a stable equilibrium existing for $\lambda < 0$ and a stable periodic orbit for $\lambda > 0,$ however the nonsmooth `canard point' is no longer a removable singularity of the reduced problem.

\begin{figure}[t]
	\begin{center}
		\includegraphics[width=0.5\textwidth]{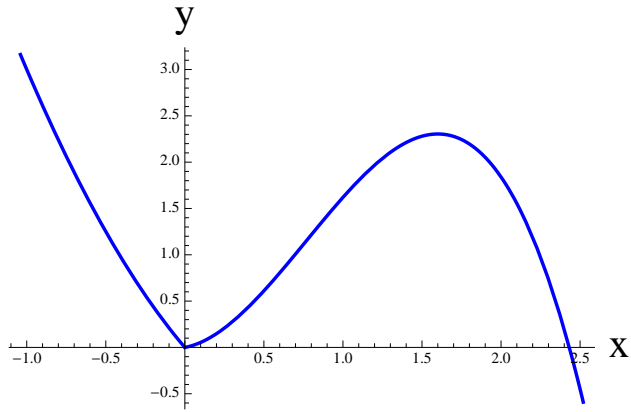}
			\caption{Example of a '2' shaped critical manifold.}
			\label{fig:2m0}
	\end{center}
\end{figure}

Nonsmooth systems are often interesting in two ways: (1) the similarities they share with smooth systems or (2) the ways they differ from smooth systems.  This paper addresses both of those issues with regard to systems of the form \eqref{genIntro}.  First, we find conditions under which \eqref{genIntro} exhibits canard phenomena similar to its smooth counterpart.  Second, we describe the dynamics when those conditions are not met.

In smooth, planar systems there have been three methods used to analyze the canard phenomenon.  It was first studied by Benoit {\it et al.} \cite{benoitEA} using nonstandard analysis.  Then Eckaus examined canards through the lens of matched asymptotic expansions \cite{eck}.  More recently, the popular mechanism for analyzing canards has been a combination of blow-up techniques and dynamical systems.  This idea was introduced by Dumortier and Roussarie \cite{dumort96} and later utilized by Krupa and Szmolyan \cite{ksGSP}, \cite{ksRO}.  Blow-up techniques are particularly important to the examination of canards in higher dimensions where canard phenomena are robust \cite{mmosurvey}.  While canard cycles only exist for an exponentially small parameter range in planar systems, the persistence of canard trajectories in higher dimensions has cast them in a new light.  This new perspective has generated increased interest in canards over the last decade.

Recently, mathematicians have begun to consider the possibility of canard-like phenomena in nonsmooth systems \cite{pwlCanards}, \cite{prohens13}, \cite{rotstein}, however the work to this point has been restricted to piecewise-linear systems.  The smooth canard explosion phenomenon involves an interplay of local dynamics near a canard point and global dynamics leading to a periodic orbit.  The examination of piecewise-linear systems is a large step in understanding the local dynamics near a nonsmooth fold, however it cannot account for the nonlinear global behavior.  Indeed, nonlinearity is required for the transition from small canard cycles to the larger canards with heads (see Figure \ref{fig:smoothCanards}).  We generalize the analysis to nonlinear, piecewise-smooth systems of the form \eqref{genIntro}. 

\begin{figure}[t]
        \centering
        \begin{subfigure}[b]{0.45\textwidth}
                \includegraphics[width=\textwidth]{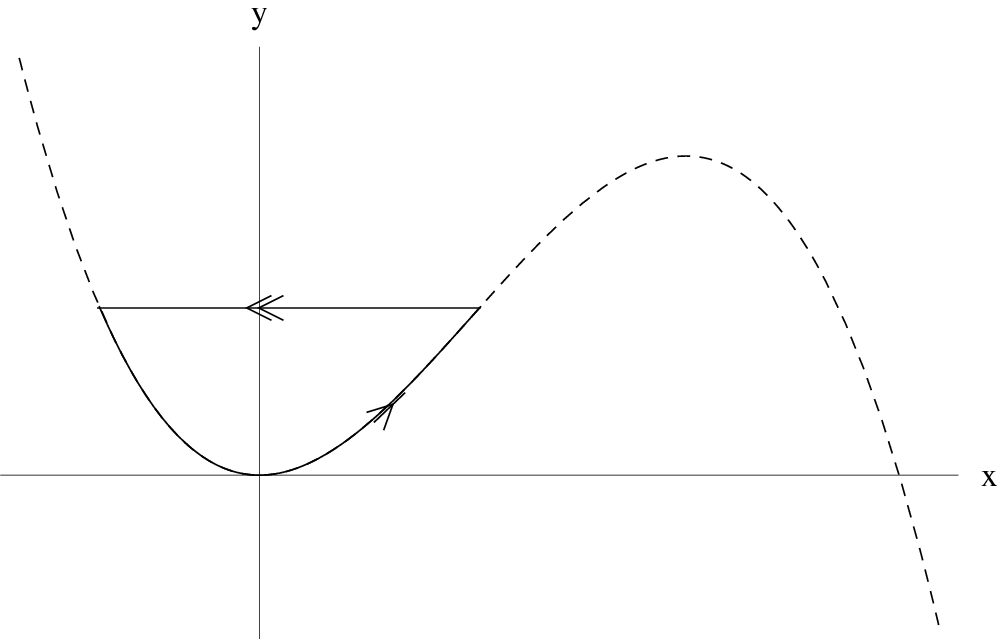}
                \caption{Singular canard cycle $\Gamma(s)$,  $s \in (0, y_M).$ }
                \label{fig:smoothNoHead}
        \end{subfigure}
        ~ 
        \begin{subfigure}[b]{0.45\textwidth}
                \includegraphics[width=\textwidth]{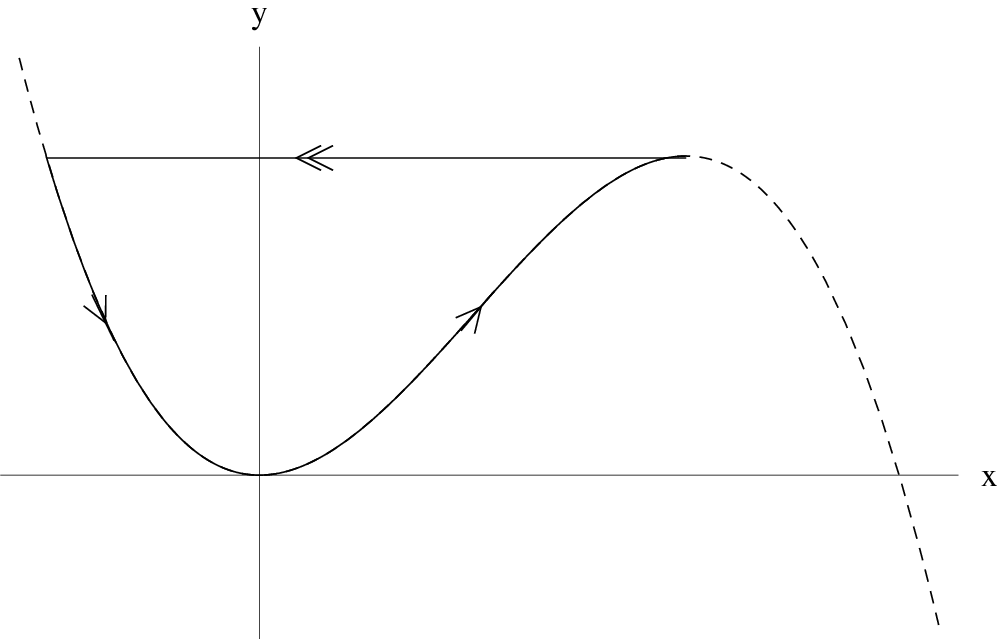}
                \caption{Singular canard cycle $\Gamma(s)$,  $s = y_M.$}
                \label{fig:smoothMax}
        \end{subfigure}
        \\ \vspace{0.4in}
                \centering
        \begin{subfigure}[b]{0.45\textwidth}
                \includegraphics[width=\textwidth]{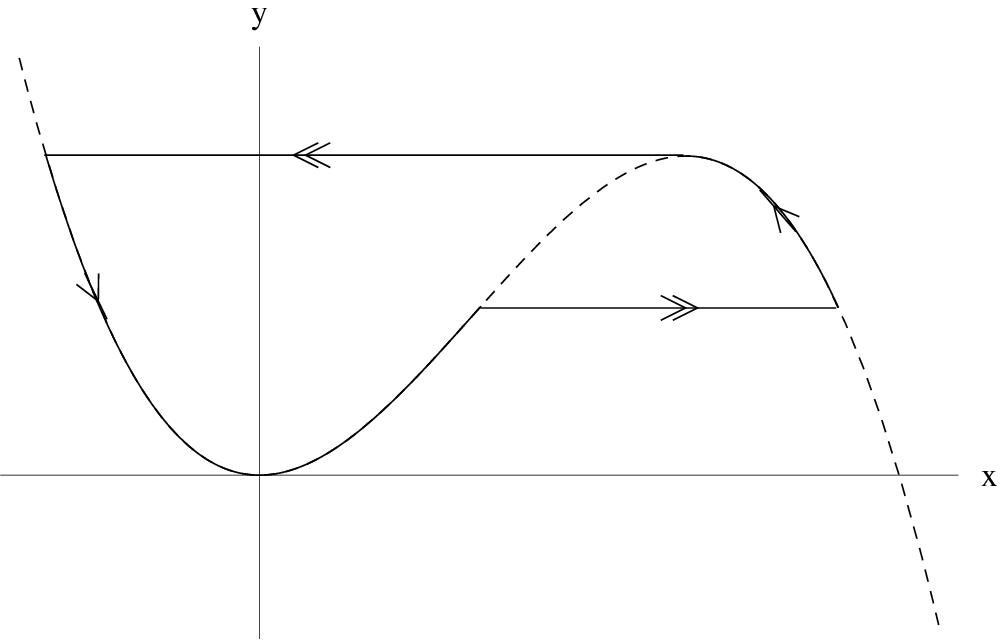}
                \caption{Singular canard cycle $\Gamma(s)$, $s \in (y_M, 2 y_M).$}
                \label{fig:smoothHead}
        \end{subfigure}
        ~ 
        \begin{subfigure}[b]{0.45\textwidth}
                \includegraphics[width=\textwidth]{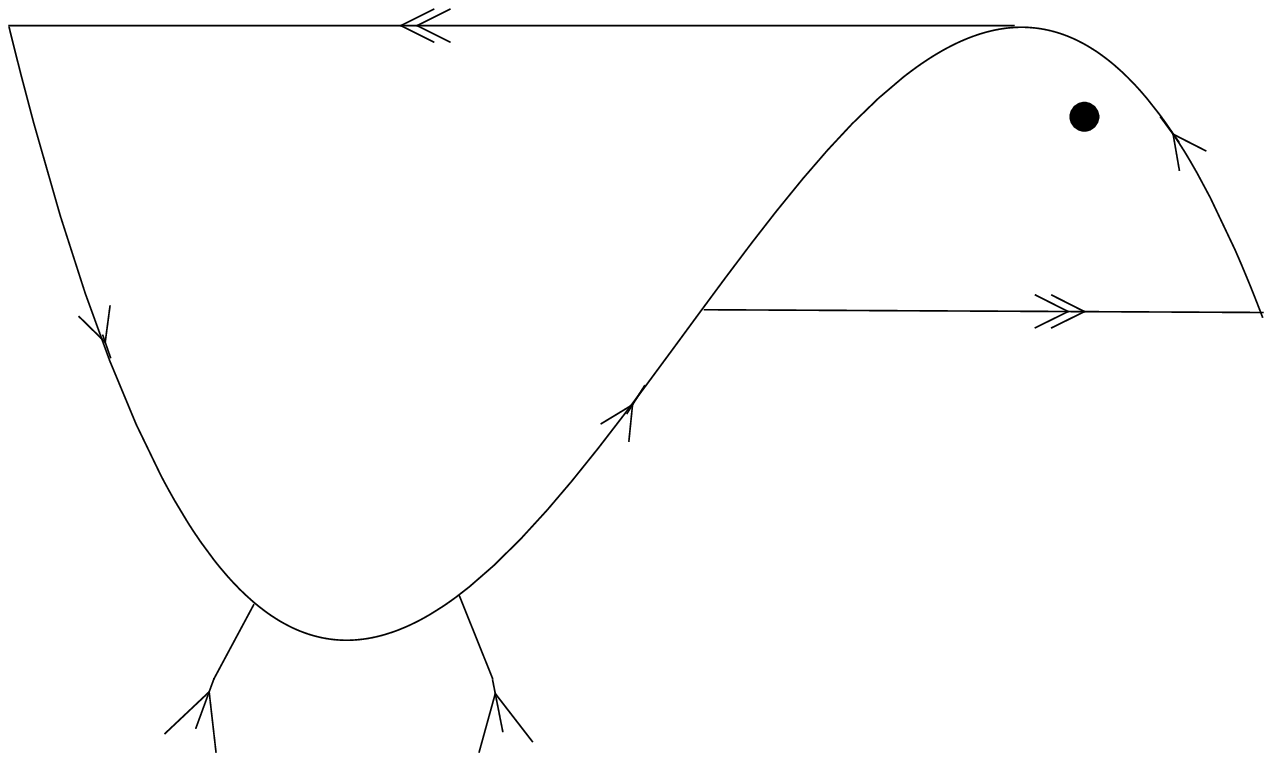}
                \caption{A duck!}
                \label{fig:smoothDuck}
        \end{subfigure}
        \caption{Examples of the three types of singular canard cycles: (a) canard without head, (b) maximal canard, and (c) canard with head.}
        \label{fig:smoothCanards}      
\end{figure}

In \cite{pwlCanards}, Desroches {\it et al.} perform the local analysis near a nonsmooth canard point in a piecewise-linear system.  We generalize their results to piecewise-smooth systems of the form \eqref{genIntro}, allowing for nonlinear effects.  Desroches {\it et al.} also discover a new phenomenon called a {\it super-explosion}, which is unique to piecewise-smooth systems.  Under a super-explosion, the system bifurcates from having a stable equilibrium instantaneously into relaxation oscillations, forgoing the small canard cycles.  We find this phenomenon in the nonlinear, piecewise-smooth case as well, however our more general setting allows us to demonstrate a sub-critical super-explosion---the simultaneous existence of a stable equilibrium and a stable relaxation oscillation.  

The method of proof employs a {\it shadow system}, or smooth system that agrees with \eqref{genIntro} on one side of the splitting line.  In most cases we will use 
\begin{equation}
	\label{shadIntro}
	\begin{array}{l}
		\dot{x} = -y + h(x) \\
		\dot{y} = \epsilon ( x - \lambda)
	\end{array}
\end{equation}
as our shadow system.  For $x > 0$, the systems \eqref{genIntro} and \eqref{shadIntro} agree.  It is often useful to consider a trajectory with initial conditions in the right half-plane, following the flow until the trajectory hits the splitting line $\{ x=0 \}$.  At this point, the vector fields no longer coincide, so the trajectory will behave differently in \eqref{genIntro} than it will in \eqref{shadIntro}.  We will compare the different behavior for $x < 0$, utilizing what is known about canard cycles in smooth systems. 

In section II we state and prove the main results about the existence or lack of canard cycles in systems of the form \eqref{genIntro}.  Section III discusses a physical application in which canard cycles arise in a nonsmooth system.  Finally, we conclude with a discussion in section IV.

\section{Results}
\subsection{Shadow system lemma}
As we state and prove our main results, we will consider systems of the form 
\begin{equation}
	\label{general}
	\begin{array}{l}
		\dot{x} = -y + F(x) \\
		\dot{y} = \epsilon ( x - \lambda)
	\end{array}
\end{equation}
where $$ F(x) = \left\{ \begin{array}{ll}
	g(x) & x\leq 0 \\
	h(x) & x \geq 0
	\end{array} \right. $$ 
with $g,h \in C^k, \ k \geq 1$, $g(0) = h(0) =0$, $g'(0) < 0$ and $h'(0) > 0$, and we assume that $h$ has a maximum at $x_M > 0$.  The critical manifold 
$$ N_0 = \{ y = F(x) \} $$ is `2'-shaped with a smooth fold at $x_M$ and a corner along the splitting line $x = 0$.   We denote
\begin{align*}
	M^l &= \{ (x, F(x) ): x<0 \} =  \{ (x, g(x) ): x<0 \} \\
	M^m &= \{ (x, F(x)) : 0 < x < x_M \} =  \{ (x, h(x)) : 0 < x < x_M \}  \\
	M^r &= \{ (x, F(x)): x > x_M \} = \{ (x, h(x)): x > x_M \}.
\end{align*}
Since $h'(0)$ exists, $h(x)$ can be extended in the left half-plane, and at least locally, we have that $h(x) \leq g(x)$ for $x<0$.  We assume that $h(x) \leq g(x)$ for all $x<0$ and define the shadow system to be
\begin{equation}
	\label{shadow}
	\begin{array}{l}
		\dot{x} = -y + h(x) \\
		\dot{y} = \epsilon ( x - \lambda).
	\end{array}
\end{equation}
Since there are two distinct types of bifurcation that can lead to the formation of periodic orbits---one at the smooth fold and one at the corner---we will consider each of those cases separately.  In both cases, we will consider the relative distance from the origin of trajectories in the nonsmooth system \eqref{general} and shadow system \eqref{shadow} that enter the left half-plane $x < 0$ at the same point $(0,y_*)$.  The following lemma describes the relationship of these trajectories.  

\begin{lemma}
	\label{bound}
	Consider the trajectory $\gamma_n (t) = (x_n (t) , y_n (t) )$ of \eqref{general} that crosses the $y$-axis entering the left half-plane $x<0$ at $\gamma_n (0) = (0,y_c)$.  Also consider the analogous trajectory $\gamma_s$ of the shadow system \eqref{shadow}.  Then, the distance from the origin of $\gamma_n$ is bounded by that of $\gamma_s$.  
\end{lemma}

\begin{proof}
Define
$$R (x,y) = \frac{x^2 + y^2}{2}.$$  
$R$ will evolve differently in the nonsmooth and shadow systems when $x < 0$, so we denote $\dot{R}_n (x,y) $ and $\dot{R}_s (x,y) $ as the time derivative of $R$ in the nonsmooth and shadow systems, respectively.  Then we have
\begin{align*}
	\dot{R}_n (x,y) &= x( g(x) - y ) + \epsilon y (x- \lambda) \\
	\dot{R}_s (x,y) &= x( h(x) - y ) + \epsilon y (x- \lambda).
\end{align*}
Therefore, at a given point $(x,y)$ we have
$$\dot{R}_n (x,y) - \dot{R}_s (x,y)  = x [ g(x) - h(x)] \leq 0, $$
since $g(x) \geq h(x)$ and $x \leq 0$, where equality only holds if $x = 0$.  Thus, $ \gamma_n$ can never cross  $\gamma_s $ moving away from the origin for $x < 0$ (i.e., the vector field of \eqref{general} points ``inward'' on the trajectory $\gamma_s$). Since $\gamma_n$ and $\gamma_s$ coincide at $(0,y_c)$, it suffices to show that $R (\gamma_n (\delta t) ) < R (\gamma_s (\delta t) )$ for $\delta t > 0$ sufficiently small.  Since the vector fields of \eqref{general} and \eqref{shadow} coincide on the $y$-axis, we must use second order terms:
\begin{align*}
	\left( \begin{array}{c}  x_n (\delta t) \\ y_n (\delta t)  \end{array} \right)&=
		\left( \begin{array}{c} 
			0 + \dot{x}_n \delta t + \ddot{x}_n \delta t^2 + \ldots \\
			y_c + \dot{y}_n \delta t + \ddot{y}_n \delta t^2 + \ldots \end{array}  \right) \\
	\left( \begin{array}{c}  x_s (\delta t) \\ y_s (\delta t)  \end{array} \right)&=
		\left( \begin{array}{c} 
			0 + \dot{x}_s \delta t + \ddot{x}_s \delta t^2 + \ldots \\
			y_c + \dot{y}_s \delta t + \ddot{y}_s \delta t^2 + \ldots \end{array}  \right).
\end{align*}

We have already seen that $(x_s(t) ,y_s (t) )$ and $(x_n (t) ,y_n (t) )$ agree for the $0^{th}$ and $1^{st}$ order terms.  Therefore, the important terms are
\begin{equation*}
	\begin{array}{lll}
		\ddot{x}_n =&	 - \dot{y}_n + g'(0) \dot{x}_n   	& = \epsilon \lambda - g'(0) y_c  		\\
		\ddot{x}_s =& 	-\dot{y}_s + h'(0) \dot{x}_s 	&= \epsilon \lambda -h'(0) y_c 			
	\end{array}
\end{equation*}		

and
\begin{equation*}
	\begin{array}{lll}
		\ddot{y}_n =& 	\epsilon \dot{x}_n  			& = -\epsilon y_c  					\\
		\ddot{y}_s =&	\epsilon \dot{x}_s 			&= - \epsilon y_c.
	\end{array}
\end{equation*}
Both vector fields point left ( $\dot{x} < 0$ ) if and only if $ y > 0$, so $y_c$ must be positive.  Since $\ddot{y}$ is the same in both systems, we look at the $\ddot{x}$ terms.  Since $-g'(0) > 0 > -h'(0)$, the smooth trajectory of the shadow system moves further left than the nonsmooth trajectory for the same vertical change as in Figure \ref{fig:arrows}.  This shows that the trajectory of the nonsmooth system enters immediately below (i.e., nearer to the origin) the trajectory of the shadow system.  Once $\gamma_n$ is nearer to the origin than $\gamma_s$, it is bounded by $\gamma_s$, proving the result.
 \begin{figure}[t]
	\begin{center}
		\includegraphics[width=0.6\textwidth]{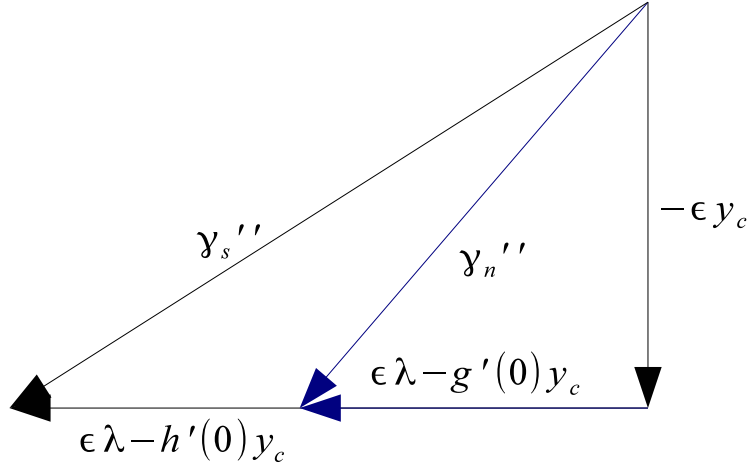}
			\caption{Relative slopes of vectors as discussed in the proof of Lemma \ref{bound}. }
			\label{fig:arrows}
	\end{center}
\end{figure}
\end{proof}

\begin{corollary}
	\label{lemCor}
	Consider the trajectory $\gamma_n (t) = (x_n (t) , y_n (t) )$ of \eqref{general} that crosses the $y$-axis entering the left half-plane $x<0$ at $\gamma_n (0) = (0,y_c)$.  Also consider the analogous trajectory $\gamma_S$ of the system
	\begin{equation}
	\label{generalS}
	\begin{array}{l}
		\dot{x} = -y + \tilde{F}(x) \\
		\dot{y} = \epsilon ( x - \lambda)
	\end{array}
\end{equation}
where $$ \tilde{F}(x) = \left\{ \begin{array}{ll}
	\tilde{g}(x) & x\leq 0 \\
	h(x) & x \geq 0
	\end{array} \right. $$ 
and $\tilde{g}'(0) > g'(0)$. Then, the radial distance from the origin of $\gamma_n$ is bounded by that of $\gamma_S$.  
\end{corollary}

\begin{proof}
The proof of Lemma \ref{bound} only used that $g'(0) < h'(0)$, not requiring that $g'(0)$ and $h'(0)$ had different signs.  Thus, using \eqref{general} as the shadow system, a similar proof to that of Lemma \ref{bound} gives the result.
\end{proof}

\subsection{Canards at the smooth fold}
\begin{theorem}
	\label{fold}
	Fix $0 < \epsilon \ll 1$.  In system \eqref{general}, assume $g(0)=0 = h(0)$, $h'(0) > 0$, and $g'(0) < 0$.  Then there is a Hopf bifurcation when $\lambda=x_M$.  
		\begin{enumerate}[(i)]
			\item If the Hopf bifurcation is supercritical, then it will produce canard cycles.  
			\item If the Hopf bifurcation is subcritical, then it will produce relaxation oscillations.
		\end{enumerate}
	\end{theorem}
	
\begin{remark}
In general, Hopf bifurcations leading to canard explosions may not occur precisely at a fold.  In our case, since the slow nullcline is a vertical line (i.e., the $\dot{y}$ equation has no $y$ dependence), the bifurcation happens when the slow nullcline intersects the fold.
\end{remark}

\begin{proof}
Direct calculation shows that \eqref{general} undergoes a Hopf bifurcation at $\lambda = x_M$.  The criticality of the Hopf bifurcation is determined as usual--the formula can be found in \cite{glendinning1994stability}, for example--since it is a smooth Hopf bifurcation.  Assume bifurcation is supercritical.  Let $\Gamma^n_\epsilon (\lambda)$ denote the family of stable periodic orbits in the nonsmooth system, and $\Gamma^s_\epsilon (\lambda)$ denote the family for the smooth system.    For some $\lambda_G(\epsilon)$, the system undergoes a grazing bifurcation where the stable periodic orbit is tangent to the splitting line $x= 1$.  If $\Gamma^s(\lambda) \subset \{ x > 0 \} $ (which happens for all $\lambda_G < \lambda < x_M$), then $\Gamma^s = \Gamma^n$.  

Beyond the grazing bifurcation (i.e., for $\lambda< \lambda_G$), $\Gamma^s$ crosses the $y$-axis transversely.  Let $y_c > 0$ be the $y$-coordinate of the crossing, and let $y_d < 0 $ be the $y$-coordinate where $\Gamma^s$ re-enters the right half-plane.  Define $\gamma_n$ to be the trajectory of \eqref{general} through the point $(0,y_d)$.  Without loss of generality, assume $\gamma_n (0) = (0,y_d)$.  Then there exists a time $t_c > 0$ so that $\gamma_n (t_c) = (0,y_c)$ and for all $t \in [ 0 , t_c ]$, $\gamma_n$ coincides with $\Gamma^s$.  By Lemma \ref{bound} for all $t > t_c$, $\gamma_n$ must be contained in the interior of $\Gamma^s$.  In particular, there exists a pair $(t_d, y_d^*)$ with $t_d > t_c$ and $y_d^* > y_d$, such that $\gamma_n(t_d) = (0,y_d^*)$.  Let $V$ be the set enclosed by the closed curve
 $$\partial V = \{ \gamma_n(t) : 0 \leq t_d \} \cup \{ (0,y): y_d^* < y_c\}$$
 
 \begin{figure}[t]
	\begin{center}
		\includegraphics[width=0.6\textwidth]{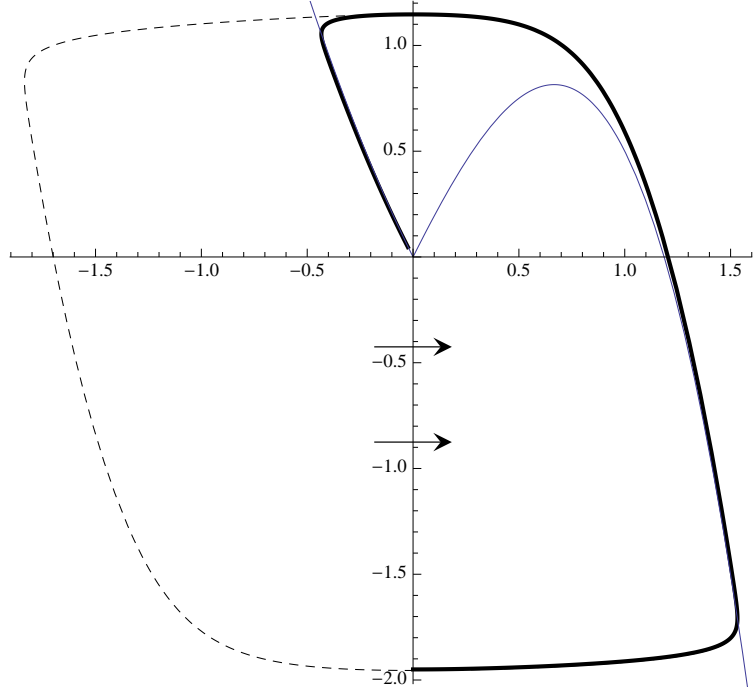}
			\caption{The set $V$ for a given $\lambda$.  The dashed curve is the periodic orbit of the shadow system.  The bold curve is the trajectory in the nonsmooth system.  $V$ is the set positively invariant set enclosed by the bold curve and the $y$-axis.}
			\label{fig:v}
	\end{center}
\end{figure}

The vector field of \eqref{general} is either tangent to or pointing inward on $ \partial V$, so the Poincar\'{e}-Bendixson theorem guarantees the existence of a stable periodic orbit on the interior of $V$ (see Figure \ref{fig:v}).  This proves (i).

The proof is the same in the case of a subcritical bifurcation, however there may not be any periodic orbits $\Gamma^s(\lambda)$ that are contained entirely in the right half-plane.
\end{proof}

There is a simple corollary to Theorem \ref{fold}.
\begin{corollary}
	\label{foldCor}
	Assume the shadow system \eqref{shadow} exhibits canard cycles (i.e., it satisfies assumptions (A1)-(A4') from \cite{ksRO}.  Then the $\Gamma^n(\lambda)$ are bounded by the stable canard orbits of the shadow system.  
\end{corollary}

Given that canard explosion happens in smooth systems, the result should not be surprising.  Perhaps more surprising is the possibility of having canard cycles arise as a result of a nonsmooth, Hopf-like bifurcation at $\lambda=0$, explained in the following theorem.  Figure \ref{fig:canSup} depicts the canard explosion at a smooth fold and a corner as a result of supercritical Hopf bifurcations in \eqref{general} where 
\begin{equation}
	\begin{array}{ccc}
	g(x) = (x-1)^2 -1 & \text{    and     } & \displaystyle h(x) = -\left(x+\frac{1}{15}\right)^2 \left(x - \frac{73}{30}\right)-\frac{73}{6750}.
	\end{array}
\end{equation}
Because programs like AUTO require smooth systems, producing diagrams like those in Figure \ref{fig:canSup} can be difficult.  Figure \ref{fig:canSupCyc} (as well as Figure \ref{fig:superExp}) was produced using NDSolve in Mathematica.  The system was simulated with different values of $\lambda$---0.01415 for the small cycle at the corner, 0.0142 for the blue orbit, 1.595 for the red orbit, and 1.596 for the small cycle at the fold.  The amplitude diagrams in Figures \ref{fig:canSup}b-d, were produced using the ode23s function in MATLAB.  The system was allowed to equilibrate to a periodic orbit for a number of fixed $\lambda$ values---with more samples being chosen in the explosion regions (depicted in Figures \ref{fig:canSupC} and \ref{fig:canSupF}).  Then the maximum and minimum $y$ coordinates of the orbit were calculated to obtain the vertical amplitude of the orbit.  

\begin{figure}[t!]
        \centering
        \begin{subfigure}[b]{0.45\textwidth}
                \includegraphics[width=\textwidth]{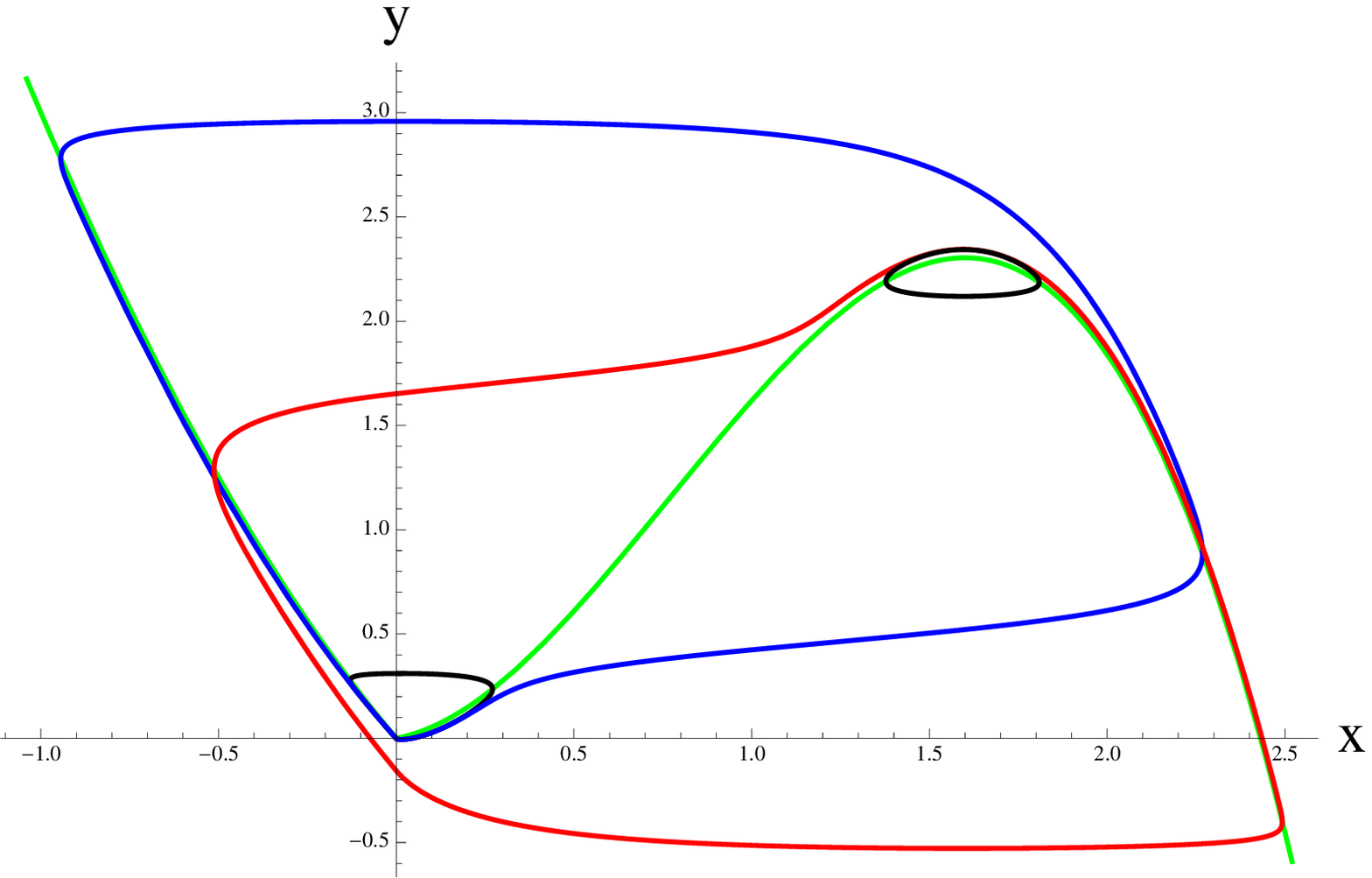}
                \caption{Nonsmooth canard cycles in the supercritical case.  The black orbits depict cycles before the explosion at both the smooth fold and the corner. The blue orbit is a post-explosion canard with head at the corner.  The red orbit is a post-explosion canard with head at the smooth fold.}
                \label{fig:canSupCyc}
        \end{subfigure}
        ~ 
        \begin{subfigure}[b]{0.45\textwidth}
                \includegraphics[width=\textwidth]{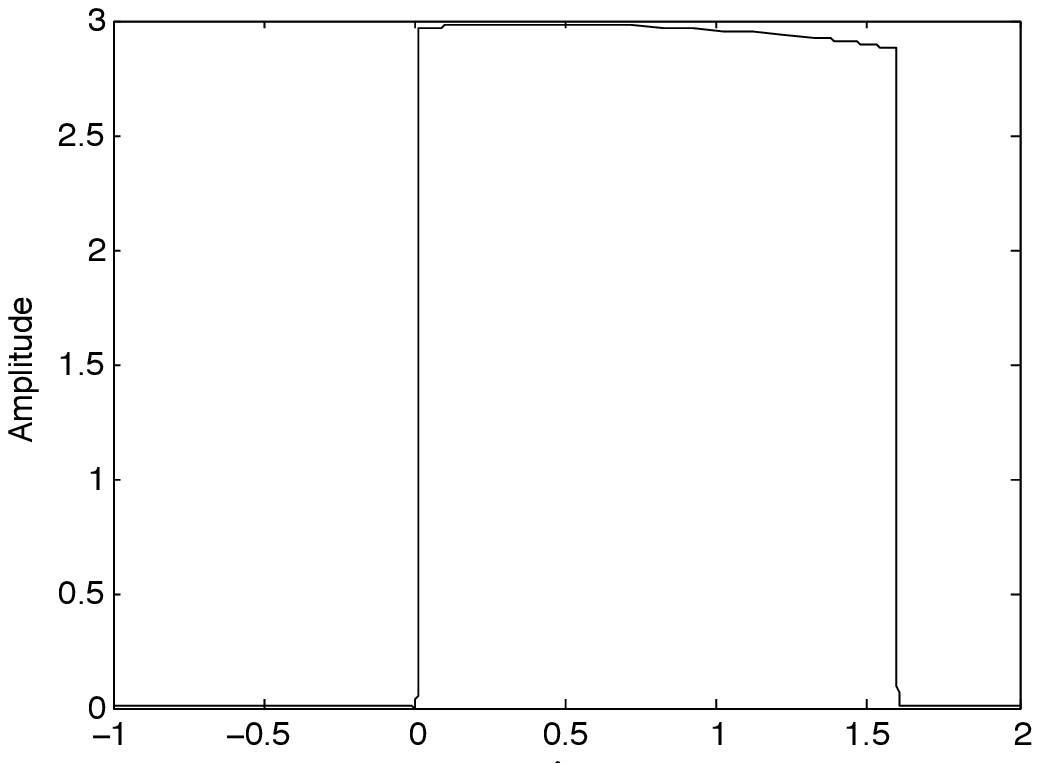}
                \caption{Amplitudes of nonsmooth canard cycles in the supercritical case, showing the explosion at the corner (on the left) and smooth fold (on the right).}
                \label{fig:canSupAmp}
        \end{subfigure}
                \centering
        \begin{subfigure}[t]{0.45\textwidth}
                \includegraphics[width=\textwidth]{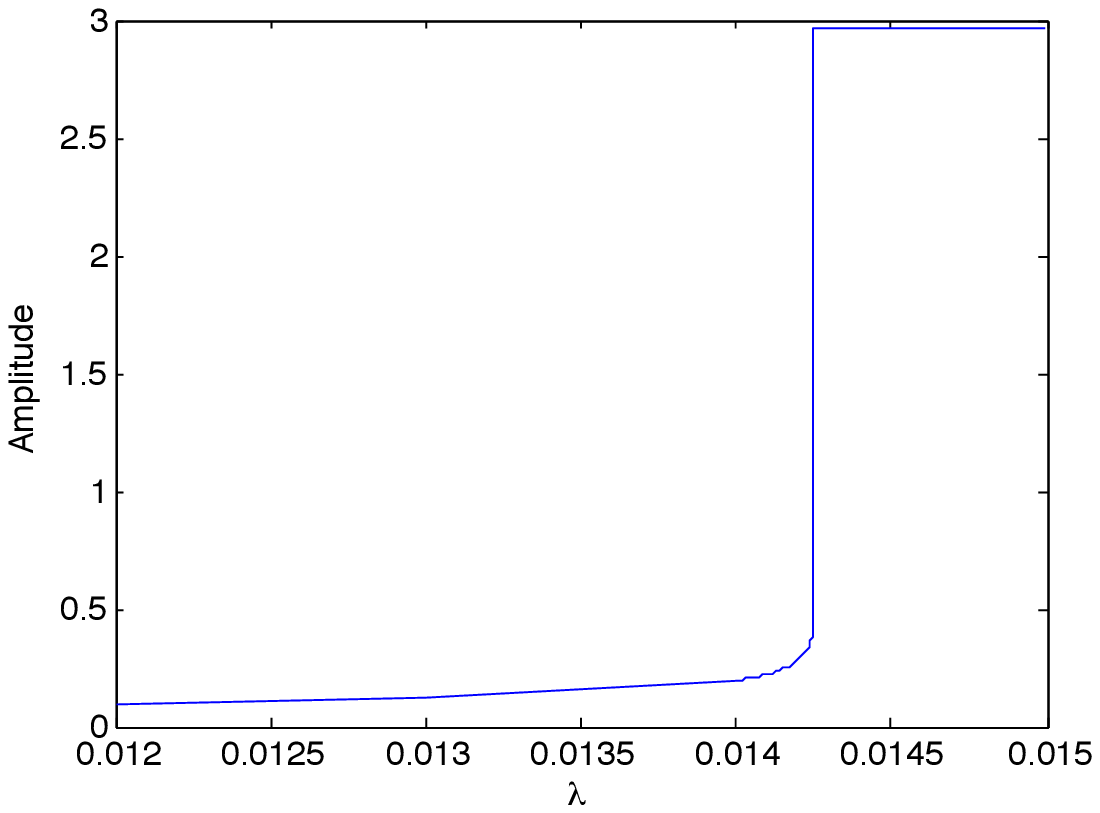}
                \caption{A closer look at the explosion at the corner.}
                \label{fig:canSupC}
        \end{subfigure}
        ~ 
        \begin{subfigure}[t]{0.45\textwidth}
                \includegraphics[width=\textwidth]{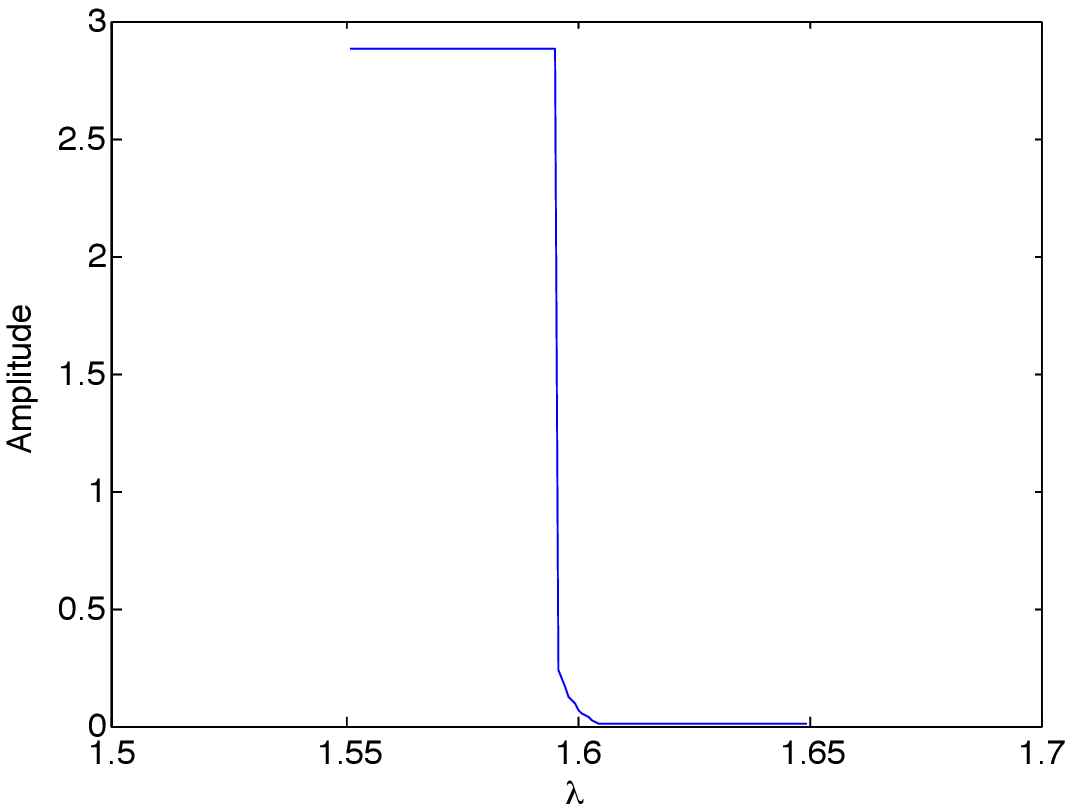}
                \caption{A closer look at the explosion at the fold.}
                \label{fig:canSupF}
        \end{subfigure}
        \caption{Images characterizing a canard explosion in non-smooth systems where $\epsilon =0.2$.  The bifurcations are supercritical at both the fold and the corner.}
        \label{fig:canSup}      
\end{figure}

\subsection{Conditions for the creation of canard cycles at the corner}
	\begin{theorem}
		\label{corner}
			In system \eqref{general}, assume $g(0)=0 = h(0)$, $h'(0) > 0$, and $g'(0) < 0$.  The system undergoes a bifurcation for $\lambda = 0$ by which a stable periodic orbit $\Gamma^n (\lambda) $ exists for $0 < \lambda < x_M.$  There exists an $\epsilon_0$ such that for all $0 < \epsilon< \epsilon_0$ the nature of the bifurcation is described by the following:
			\begin{enumerate}[(i)]
				\item If $0 < h'(0) < 2 \sqrt{\epsilon}$, then periodic orbits $\Gamma^n(\lambda)$ are born of a Hopf-like bifurcation as $\lambda$ increases through 0.  The bifurcation is subcritical if $| g'(0) | < | h'(0) |$ and supercritical if $| g'(0) | > | h'(0) |.$ 				
				\item If $h'(0) > 2 \sqrt{\epsilon}$, the bifurcation at $\lambda=0$ is a super-explosion.  The system has a stable periodic orbit $\Gamma^n(\lambda)$, and $\Gamma^n (\lambda)$  is a relaxation oscillation.  If $ |g'(0)| \geq 2\sqrt{\epsilon}$, the bifurcation is supercritical in that no periodic orbits appear for $\lambda<0$.  However, if $ |g'(0)| < 2 \sqrt{\epsilon}$ the bifurcation is subcritical, in that a stable periodic orbit and stable equilibrium coexist simultaneously for some $\lambda < 0$.
			\end{enumerate}
	\end{theorem}

\begin{proof}
The system always has a unique equilibrium $p_\lambda = (\lambda, F(\lambda))$, and direct computation of the corresponding eigenvalues shows that
	\begin{equation}
		\mu_{\pm} (\lambda) = \left\{ \begin{array}{ll}
			\displaystyle \frac{g'(\lambda) \pm \sqrt{ [g'(\lambda)]^2 - 4 \epsilon} }{2} ,	 & \lambda<0 \\ [\bigskipamount]
			\displaystyle \frac{h'(\lambda) \pm \sqrt{ [h'(\lambda)]^2 - 4 \epsilon} }{2} ,	 & \lambda > 0 .
			\end{array} \right. 
	\end{equation}

For $0 < \lambda < x_M$ we have $h'(\lambda) > 0$, so both eigenvalues have $\text{Re}(\mu_{\pm}) > 0$ and the equilibrium $p_\lambda$ is unstable.  To demonstrate the existence of a stable periodic orbit $\Gamma^n(\lambda)$, it suffices to show that there is a positively invariant set containing $p_\lambda$.  Consider the set $W$ as shown in Figure \ref{fig:w}.  The boundary $\partial W$ is composed of six line segments $l_j$ for $j = 1, \ 2,\dots 6.$  Choose any point $(\hat{x}, 0 )$ with $\hat{x} < 0$.  

\begin{figure}[t]
	\begin{center}
		\includegraphics[width=0.9\textwidth]{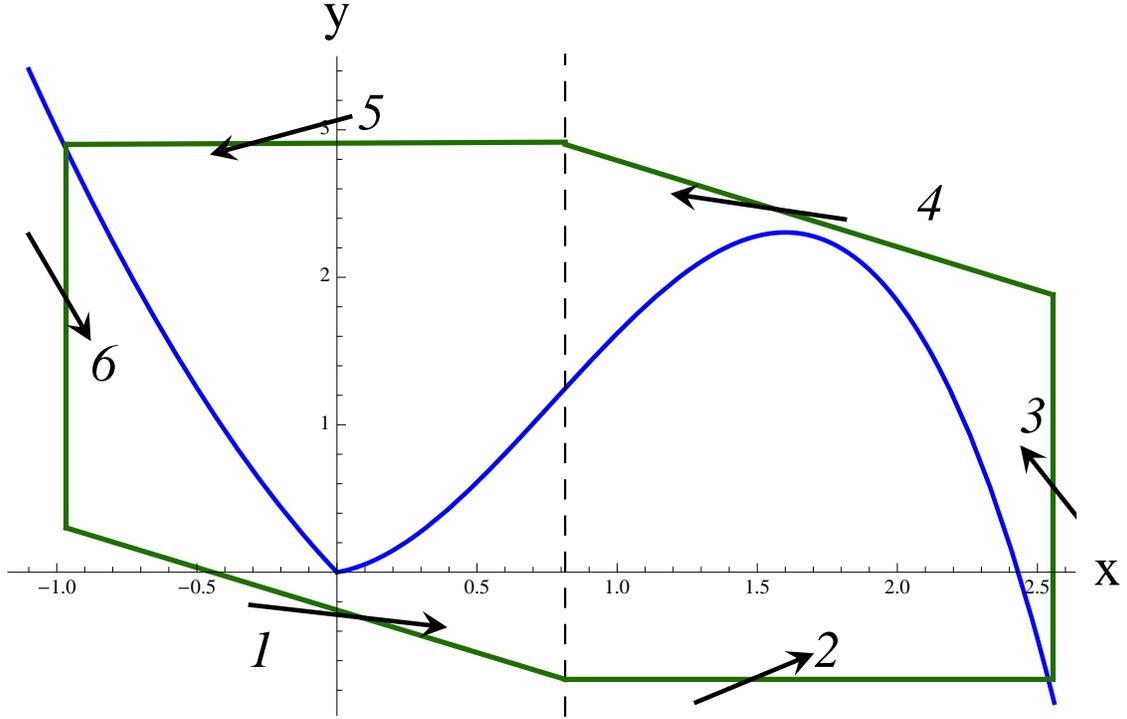}
			\caption{An example of a set $W$ which is positively invariant.  The numbers indicate the 6 line segments $l_i$ forming the boundary $\partial W$.}
			\label{fig:w}
	\end{center}
\end{figure}

\begin{align*}
	l_1 &= \{ ( x , y ) : y = m_1 (x- \hat{x} ), \ m_1 < 0 \text{ is } \mathcal{O}(1), \ \hat{x} \leq x \leq \lambda \} \\
	l_2 &= \{ (x , y_1 ): y_1 = m_1 (\lambda - \hat{x}),  \lambda \leq x \leq x_2 = (h^{-1}(y_1) + 1) \} \\
	l_3 &= \{ (x_2 , y ): y_1 < y < y_3 = h(x_M) \} \\
	l_4 &= \left\{ ( x , y ) : y=m_4(x-x_2)+y_3 \text{ where } m_4 > \frac{ g ( \hat{x} )- y_3 }{ \lambda - x_2 } \text{ is } \mathcal{O}(1) \right\} \\
	l_5 &= \{ ( x , y_5 ) : y_5 = m_4(\lambda - x_2)+y_3, \ \hat{x} < x < \lambda \} \\
	l_6 &= \{ (\hat{x} , y) : 0<y<y_5 \}
\end{align*}
Note that $W$ can be made as large as needed.  The vector field points inward on $\partial W$ and the Poincar\'{e}-Bendixson theorem guarantees the existence of an attracting periodic orbit $\Gamma^n (\lambda)$ for $0<\lambda < x_M$.  The $\Gamma^n$ created through the bifurcation at $\lambda=0$ will differ in amplitude depending on $h'(0)$.

First, we consider the case where $0 < h'(0) < 2 \sqrt{\epsilon}$.  If $| g'(0) |  < 2 \sqrt{\epsilon} $ as well, the bifurcation at $\lambda=0$ is a nonsmooth Hopf bifurcation as described by Simpson and Meiss \cite{simpmeiss}.  We define		\begin{align*}
		\mu_u &= h'(0) \\
		\mu_s &= |g'(0)| \\
		\omega_u &= \sqrt{ (4 \epsilon - [h'(0)]^2)} \\
		\omega_s &= \sqrt{ (4 \epsilon - [g'(0)]^2 )},
	\end{align*}
	and
	$$ \Lambda = \frac{ \mu_u}{\omega_u} - \frac{ \mu_s}{ \omega_s}.$$
According to \cite{simpmeiss}, the nonsmooth Hopf bifurcation is subcritical if $\Lambda <0$ and supercritical if $\Lambda > 0.$  We have
	\begin{align*}
	\Lambda < 0 	&\iff  		\frac{ \mu_u}{\omega_u} < \frac{ \mu_s}{ \omega_s} \\
				&\iff		h'(0) \sqrt{ [4 \epsilon - g'(0)]^2 } < |g'(0)| \sqrt{  4 \epsilon - [h'(0)]^2} \\
				&\iff		4\epsilon [h'(0)]^2 -[h'(0)]^2 \cdot  [g'(0)]^2   < 4 \epsilon [g'(0)]^2  [h'(0)]^2 \cdot  [g'(0)]^2  \\
				&\iff		 [h'(0)]^2 < [g'(0)]^2.
	\end{align*}
Therefore, there is a subcritical nonsmooth Hopf bifurcation when $h'(0) < |g'(0)| $ and supercritical nonsmooth Hopf bifurcation when $ h'(0) > |g'(0)| $.  Corollary \ref{lemCor}, also guarantees the existence of canard cycles for the case where $|g'(0)| > 2 \sqrt{\epsilon}$.  The bifurcation in that case is a stable node-to-unstable focus.  The canard cycles in the system with a node will be contained in the cycles for the system with a nonsmooth Hopf (stable focus-to-unstable focus) bifurcation.  This proves assertion $(i)$.

Next, we consider the case where $h'(0) \geq 2 \sqrt{\epsilon}.$  In this case, for $\lambda > 0$ (but bounded away from $x_M$), the equilibrium $p_\lambda$ is an unstable node.  Let $\mu_2 \geq \mu_1> 0$ be the strong and weak unstable eigenvalues corresponding to $p_\lambda$.  Also, let $v_{1,2} = ( x_{1,2},y_{1,2} )$ be the associated eigenvectors.  Then for $i =1 , 2$ we have
\begin{equation*}
	\begin{array}{l}
		h'(\lambda) x_i - y_i = \mu_i x_i \\
		\epsilon x_i = \mu_i y_i.
	\end{array}
\end{equation*}
This implies that the slope of the eigenvector
$$v_2 =   \frac{ \epsilon}{ \mu_2}.$$
Now, $\mu_2$ depends on $\epsilon$, however $\mu_2 \rightarrow 2 h'(\lambda) \neq 0$ as $\epsilon \rightarrow 0$.  Thus $v_2 \rightarrow 0$ as $\epsilon \rightarrow 0$.  Therefore, there exists an $\epsilon_0$ such that for all $0< \epsilon < \epsilon_0$, the strong unstable trajectory must enter the region of phase space $x > x_M$, pass over the point $(x_M, h (x_M))$ and proceed to the left-half plane.  Since the vector field of \eqref{general} points into the left half-plane along the $y$-axis for $y > 0$, the strong unstable trajectory must re-enter the right half-plane somewhere below $y = 0$.  Following the trajectory further, it must continue downward to the right until it reaches the $y$-nullcline $x=\lambda$ at some point $(\lambda, \hat{y})$ where $\hat{y} < h(\lambda) $.  Let $V$ be the region enclosed by the trajectory described above and the line segment along $x=\lambda$ connecting $\hat{y}$ and $h(\lambda)$.  All trajectories in $W$ that start outside of $V$ are bounded outside of $V$ (see Figure \ref{fig:vSupFig}).  Therefore $\Gamma^n (\lambda)$ must be a relaxation oscillation, and the bifurcation must be a super-explosion, as depicted in Figure \ref{fig:superExp}.  

\begin{figure}[t]
	\begin{center}
		\includegraphics[width=0.5\textwidth]{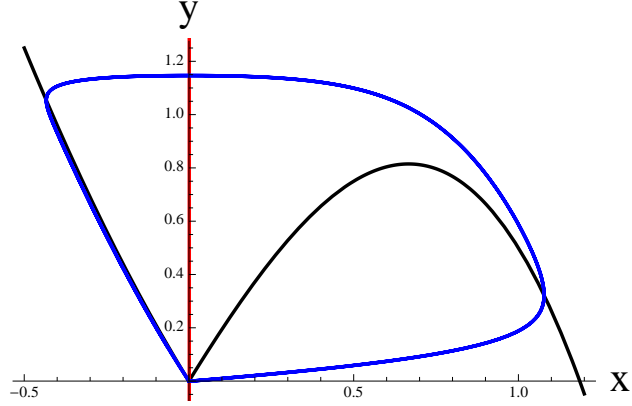}
			\caption{The stable orbit of a super-explosion (blue).  The system used is \eqref{general} with $g(x) = (x-1)^2-1$, and $h(x) = -(x+1)^2 (x-1.5) - 1.5$ with $\epsilon = 0.2$ and $\lambda = 0.001$.  The line $x=\lambda$ (red) is the slow nullcline. }
			\label{fig:superExp}
	\end{center}
\end{figure}

Suppose also that $g'(0) \geq 2 \sqrt{\epsilon}$.  Then, for $\lambda < 0$ the equilibrium $p_\lambda$ is a stable node.  It is the global attractor of the system (similar to the piecewise-linear case in \cite{pwlCanards}) since, the strong stable trajectory to $p_\lambda$ bounds trajectories above the node in the left half-plane.  Since no stable periodic orbits can coexist with $p_\lambda$, we say the bifurcation is supercritical.

On the other hand, if $g'(0) < 2 \sqrt{\epsilon}$, $p_\lambda$ is a stable focus.  For $\lambda < 0$, $| \lambda |$ sufficiently small, there exists a $\beta \in ( g(\lambda) , h(x_M) )$ such that the trajectory through $(0, \beta)$ spirals around $p_\lambda$ and enters the right half plane below $(0,0)$.  The dynamics in the right half-plane are governed by an unstable node in the shadow system \eqref{shadow}.  After entering the right half-plane, the trajectory through $\beta$ must proceed to cross $M^r$ and reenter the left half-plane at $(0, \beta ')$ where $\beta' > \beta$.  If we let $V'$ denote the region enclosed by the trajectory through $(0,\beta)$ and the line segment along the $y$-axis connecting $\beta$ to $\beta'$, then $V'$ is a negatively invariant set.  We can choose $W$ large enough so that $V' \subset W$.  The set $W \setminus V'$ is positively invariant and contains no stable equilibria (see Figure \ref{fig:vPrime}).  Therefore, there must be an attracting periodic orbit $\Gamma^n(\lambda)$ inside $W \setminus V'$.  Since an attracting equilibrium and an attracting periodic orbit coexist simultaneously, we call this a subcritical super-explosion.  This proves assertion (ii).
\end{proof}

\begin{figure}[t!]
        \centering
        \begin{subfigure}[t]{0.45\textwidth}
                \includegraphics[width=\textwidth]{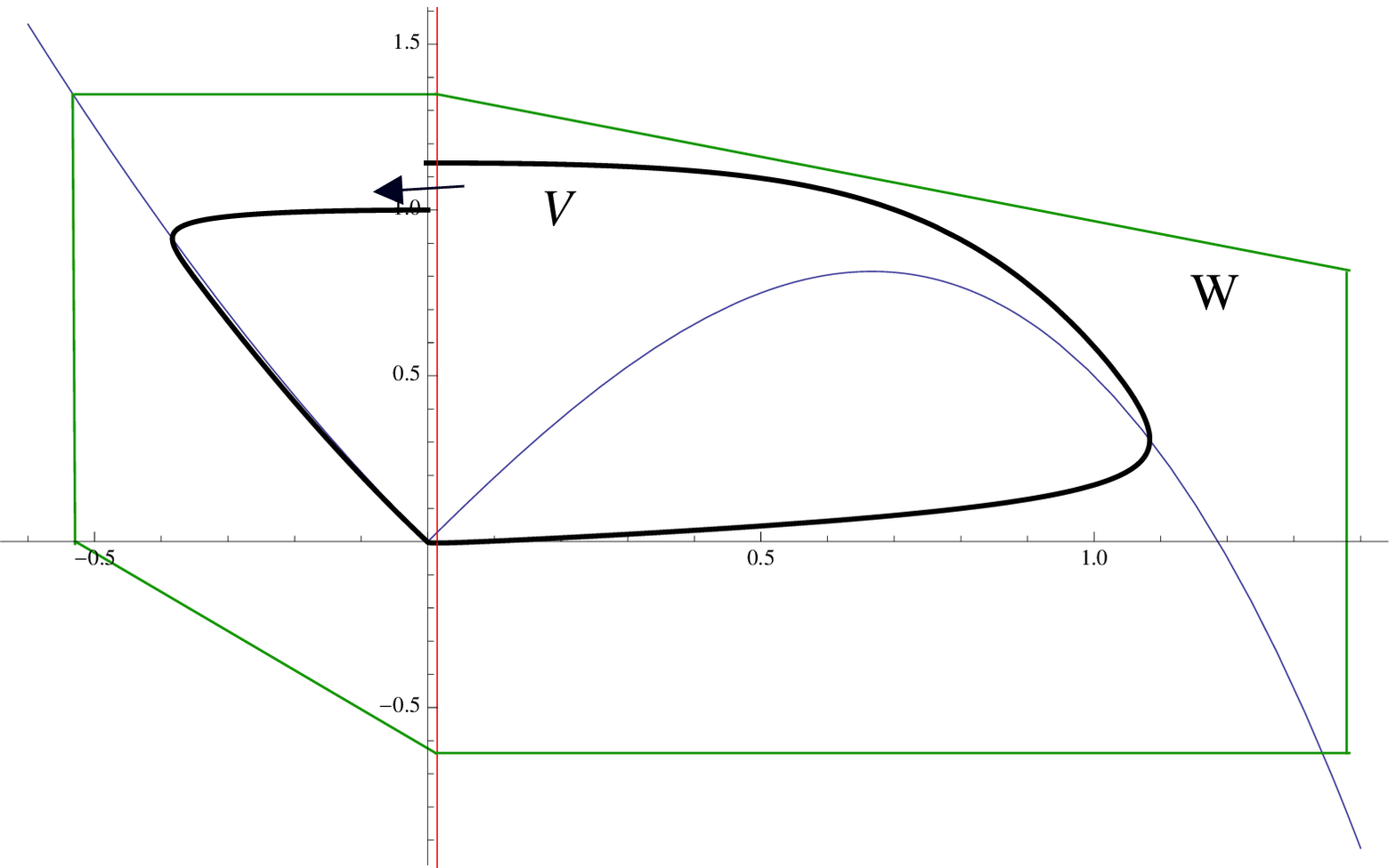}
			\caption{The set $ W \setminus V$ for $\lambda = 0.014$.}
			\label{fig:vSupFig}
        \end{subfigure}
        ~ 
        \begin{subfigure}[t]{0.45\textwidth}
                \includegraphics[width=\textwidth]{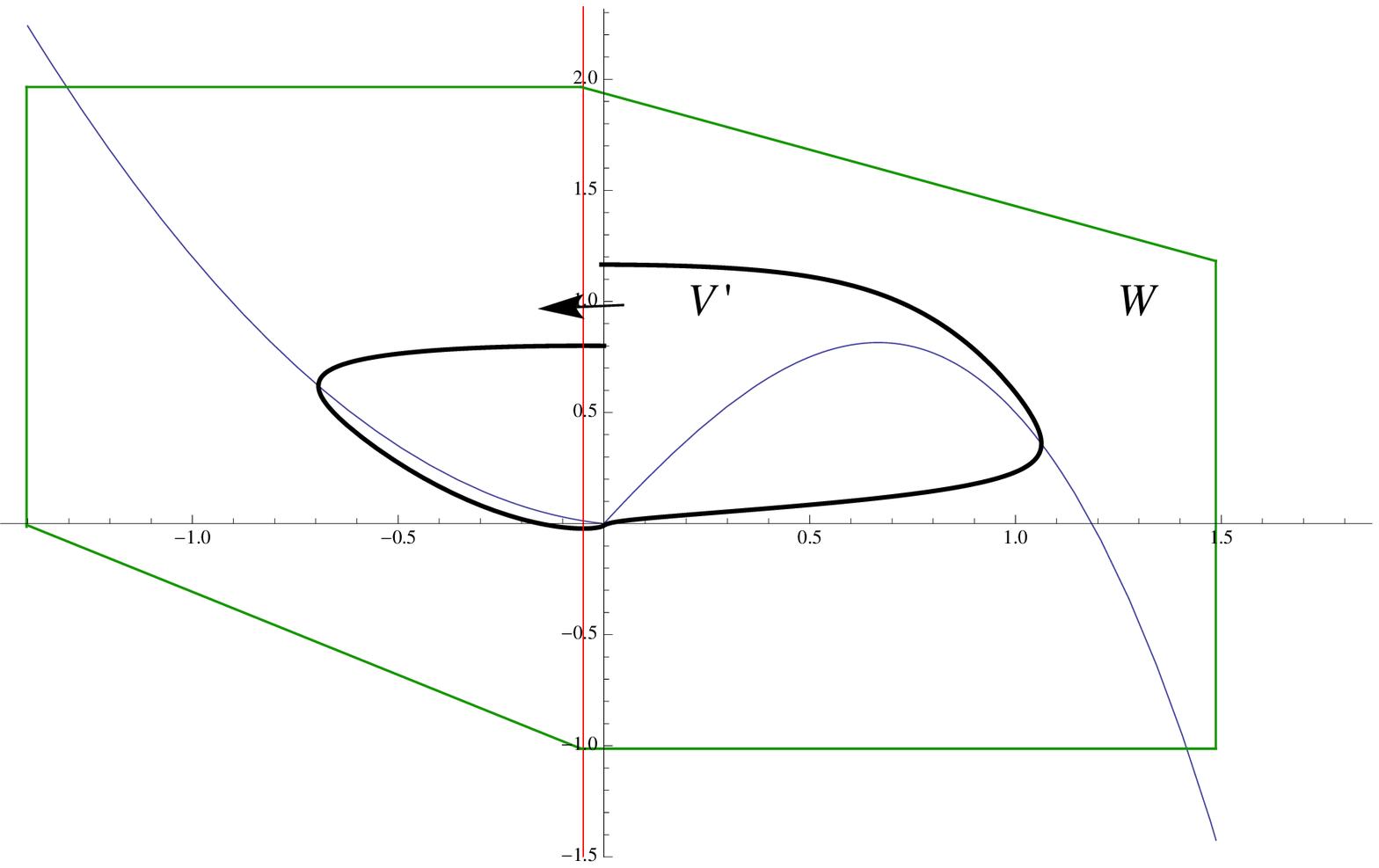}
			\caption{The set $ W \setminus V'$ for $\lambda = -0.05$.}
			\label{fig:vPrime}
        \end{subfigure}
        \caption{Positively invariant sets demonstrating the existence of attracting periodic orbits for (a) the standard (supercritical) super-explosion and (b) the subcritical super-explosion.  $W$ is the region bounded by the six (green) line segments as in Figure \ref{fig:w}.  The sets enclosed by the bold trajectory are (a) $V$ or (b) $V'$. }
        \label{fig:}      
\end{figure}

There are two simple corollaries to Theorem \ref{corner}.  The first corollary treats the case of a `Z'-shaped critical manifold.
	\begin{corollary}
		\label{corner2}
			Consider the system with two corners and no smooth folds:
\begin{equation}
	\label{general2}
	\begin{array}{l}
		\dot{x} = -y + F(x) \\
		\dot{y} = \epsilon ( x - \lambda)
	\end{array}
\end{equation}
where $$ F(x) = \left\{ \begin{array}{ll}
	g(x) & x\leq 0 \\
	h(x) & 0 \leq x \leq x_M\\
	f(x) & x_M \leq x
	\end{array} \right. $$ 
with $g,h,f \in C^k, \ k \geq 1$, $g(0) = h(0) =0$, $f(x_M) = h(x_M)$ $g'(0) < 0$, $h'(x) > 0 $ for all $0< x < x_M$, and $f'(x_M) < 0$.  
			The system undergoes a bifurcation for $\lambda = 0$ by which a stable periodic orbit $\Gamma^n (\lambda) $ exists for $0 < \lambda < x_M.$  There exists an $\epsilon_0$ such that for all $0 < \epsilon< \epsilon_0$ the nature of the bifurcation is described by the following:
			\begin{enumerate}[(i)]
				\item If $0 < h'(0) < 2 \sqrt{\epsilon}$, then canard cycles $\Gamma^n(\lambda)$ are born of a Hopf-like bifurcation as $\lambda$ increases through 0.  The bifurcation is subcritical if $| g'(0) | < | h'(0) |$ and supercritical if $| g'(0) | > | h'(0) |.$ 				
				\item If $h'(0) > 2 \sqrt{\epsilon}$, the bifurcation at $\lambda=0$ is a super-explosion.  The system has a stable periodic orbit $\Gamma^n(\lambda)$, and $\Gamma^n (\lambda)$  is a relaxation oscillation.  If $ |g'(0)| \geq 2\sqrt{\epsilon}$, the bifurcation is supercritical in that no periodic orbits appear for $\lambda<0$.  However, if $ |g'(0)| < 2 \sqrt{\epsilon}$ the bifurcation is subcritical, in that a stable periodic orbit and stable equilibrium coexist simultaneously for some $\lambda < 0$.
			\end{enumerate}
	\end{corollary}
\begin{proof}
	The proof is the same as that of Theorem \ref{corner}, only we use \eqref{general} as our shadow system.
\end{proof}

The second corollary is an immediate application of Lemma \ref{bound}.
	\begin{corollary}
		\label{cornerBound}
		Fix $\epsilon > 0$. Assume the shadow system \eqref{shadow} satisfies the assumptions (A1)-(A4') from \cite{ksRO} for a canard point at $(x_m,h(x_m))$ where $x_m < 0$.  Also assume that $0 < h'(0) < 2 \sqrt{\epsilon}$ in the nonsmooth system \eqref{general}.  Then for fixed $\lambda \in (0, x_M)$, $\Gamma^n (\lambda)$ is bounded by the periodic orbit $\Gamma^s$ of the shadow system.  Furthermore, if the criticality parameter $A < 0$ and $ |x_m| < \lambda_s(\sqrt{\epsilon})$, then the periodic orbits $\Gamma^n (\lambda)$ created through the Hopf-like bifurcation will be canard cycles and the system \eqref{general} will undergo a canard explosion.
	\end{corollary}
	
We have demonstrated that there are two types of nonsmooth bifurcations in which periodic orbits appear before the parameter reaches the bifurcation value.  We will show that these are truly subcritical bifurcations.  That is, as the bifurcation parameter moves away from the bifurcation value, the periodic orbits are destroyed.

\begin{proposition}
	\label{subcritDecay}
	In a system of the form \eqref{general} suppose there exists an $m < 0$ such that $g'(x) \leq m  < 0$ for all $x < 0$.  Then there exists a $K > 0$ such that if $\lambda < - K$, the system has no periodic orbits.
\end{proposition}

\begin{proof}
The idea of the proof amounts to using a variation of Dulac's criterion \cite{glendinning1994stability} for the non-existence of periodic orbits.  Define 
$$G (x,y) = ( F(x) - y, \epsilon( x - \lambda)), $$
so $G(x,y)$ is the vector field.  We will need to use the divergence $\nabla \cdot G$ often, and direct computation shows
$$\nabla \cdot G = F'(x).$$

First, if $\lambda < 0$, the only equilibrium lies in the left half-plane.  Since any periodic orbit of a planar system must encircle an equilibrium, there can be no periodic orbits entirely contained in the set $x \geq 0$.

Second, there can be no periodic orbits entirely contained in the left-half plane.  We will show this by contradiction.  Suppose there is a periodic orbit $\Gamma$, and define $D$ to be the region enclosed by $\Gamma$.  Then
$$ \nabla \cdot G \leq m < 0 \hspace{0.5in} \text{ for all } x < 0.$$
Therefore,
$$ \int \int_D \nabla \cdot G dx dy <  0. $$
By the divergence theorem $$\int_\Gamma (n \cdot F) ds = \int \int_D \nabla \cdot F dx dy.$$
However, $\Gamma$ is a trajectory, so
$$ \int_\Gamma (n \cdot G) ds = 0,$$ and we have a contradiction.

We will show that there are no periodic orbits that cross $x = 0$ in a similar way.  Suppose there is a periodic orbit which crosses $x=0$.  Then it must do so twice; let $p_1 = (0,y_1)$, $p_2 = (0,y_2)$ with $y_1(\lambda) > 0 > y_2 (\lambda) $ be the points where $\Gamma$ intersects the $y$-axis.  Also, define $B (\lambda) = y_1 (\lambda) - y_2 (\lambda)$.  Note that $B(\lambda)$ is the maximum vertical amplitude of the periodic orbit in the region $x \geq 0$.  Let $k$ be the maximum slope of $h(x)$ on the interval $0 \leq x \leq x_M.$  Since $F'(x) \geq 0$ on the set $x \in [0, x_M]$, we know
$$ \int \int_{D \cap \{ x \geq 0 \} } \nabla \cdot G dx dy \leq \int \int_{R(\lambda)} \nabla \cdot G dx dy \leq k x_M B(\lambda),$$
where $R(\lambda)$ is the rectangle with $B(\lambda)$ forming the left side and having width $x_M.$
Next, let $p_0 = (x_0,y_0)$ be the point where $\Gamma$ intersects the fast nullcline $y = g(x)$, for $x_0 < \lambda$.  Define $B_1(\lambda)$ to be the line segment connecting $p_0$ to $p_1$ and $B_2 (\lambda)$ to be the line segment connecting $p_0$ to $p_2$.  Then $B_1$ and $B_2$ have constant slope (for fixed $\lambda$).  Let $D^*(\lambda)$ be the interior of the triangle enclosed by $B(\lambda),$ $B_1(\lambda)$, and $B_2(\lambda)$.  Then $D^*$ must lie entirely inside $D$.  Along $B_1$ near $p_1$, the vector field must point outside of $D^*$.  Since the slopes of the vectors are monotonically decreasing along $B_1$, if $\Gamma$ were ever to cross $B_1$ somewhere other than the endpoints, $\Gamma$ would be trapped inside $D^*$.  This would contradict that $p_0$ lies on $\Gamma.$  A similar argument in reverse time shows that $\Gamma$ cannot cross $B_2$.  Figure \ref{fig:dulac} shows the sets $D^*(\lambda)$ and $R(\lambda)$.

\begin{figure}[t]
	\begin{center}
		\includegraphics[width=0.5\textwidth]{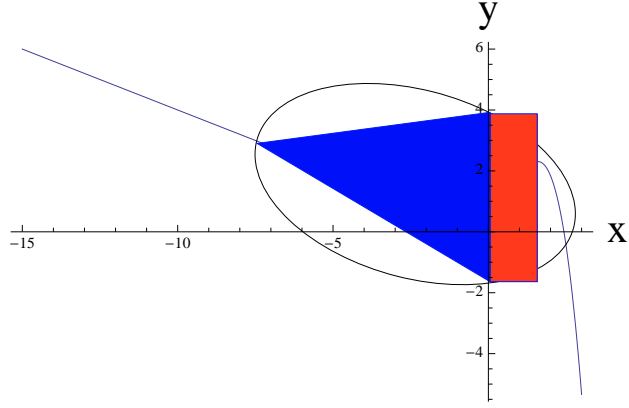}
			\caption{Important sets for the proof of Proposition \ref{subcritDecay}.  $\nabla \cdot G < 0$ is negative on the interior of the triangle $D^*(\lambda)$ (blue).   $\left. \nabla \cdot G\right|_D > 0$ on a region bounded by the rectangle $R(\lambda)$ (red).  }
			\label{fig:dulac}
	\end{center}
\end{figure}

If we let $A(\lambda)$ be the area of the region $D^*(\lambda)$, then
$$ A(\lambda) > \frac{ | \lambda |}{2} B(\lambda).$$
Since $g'(x) \leq m < 0 $, we have 
$$\int \int_{D \cap \{ x < 0 \} } \nabla \cdot G dx dy < \int \int_{D^*} \nabla \cdot G dx dy < m \frac{\lambda}{2} B(\lambda).$$
Thus, if $\lambda < 2 x_M \frac{k}{m}$, we can conclude
$$\int \int_D G dx dy < 0.$$  Therefore, by the divergence theorem, $\Gamma$ cannot be a periodic orbit.
\end{proof}

\section{Application in Ocean Circulation}
One might wonder if there are physical systems in which knowledge of canard orbits at a corner is useful.  We present here a variation of Stommel's thermohaline circulation model \cite{stommel}.  Stommel's original model from 1961 contains a '2'-shaped bifurcation curve producing two saddle-node bifurcations: one at a corner and one a fold.  Incorporating a parameter into the model as a slow state variable transforms the model into a fast/slow system in the form of \eqref{genIntro}.  It seems likely that systems with similar nonsmooth saddle-node bifurcations could be approached in the same way, where some combination of Theorem \ref{fold}, Theorem \ref{corner}, and Corollary \ref{corner2} apply.  


In \cite{kg11} it is shown that the dimensionless Stommel model reduces to the simple nonsmooth system
	\begin{equation}
		\label{red1}
		\dot{y} = \mu - y - K | 1- y | y,
	\end{equation}
which has equilibria at
	\begin{equation}
		\label{crit1}
		\mu = \left\{ \begin{array}{ll}
			(1 + K) y - K y^2 & \text{for } y < 1 \\
			(1 - K) y + K y^2 & \text{for } y > 1
			\end{array} \right.
	\end{equation}
The nature of the system depends on $K$, as seen in Figure \ref{fig:As}. 

 \begin{figure}[b]
        \centering
        \begin{subfigure}[b]{0.45\textwidth}
		\includegraphics[width=\textwidth]{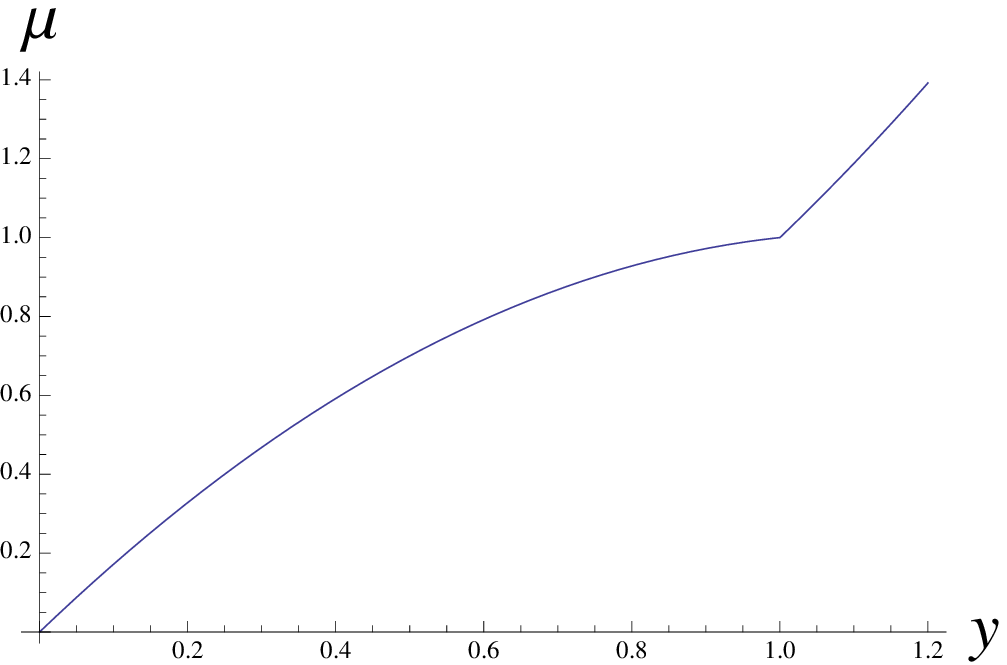} 
		\caption{$K < 1$}      
		 \label{fig:smallA}
	\end{subfigure}
        ~ 
        \begin{subfigure}[b]{0.45\textwidth}
		\includegraphics[width=\textwidth]{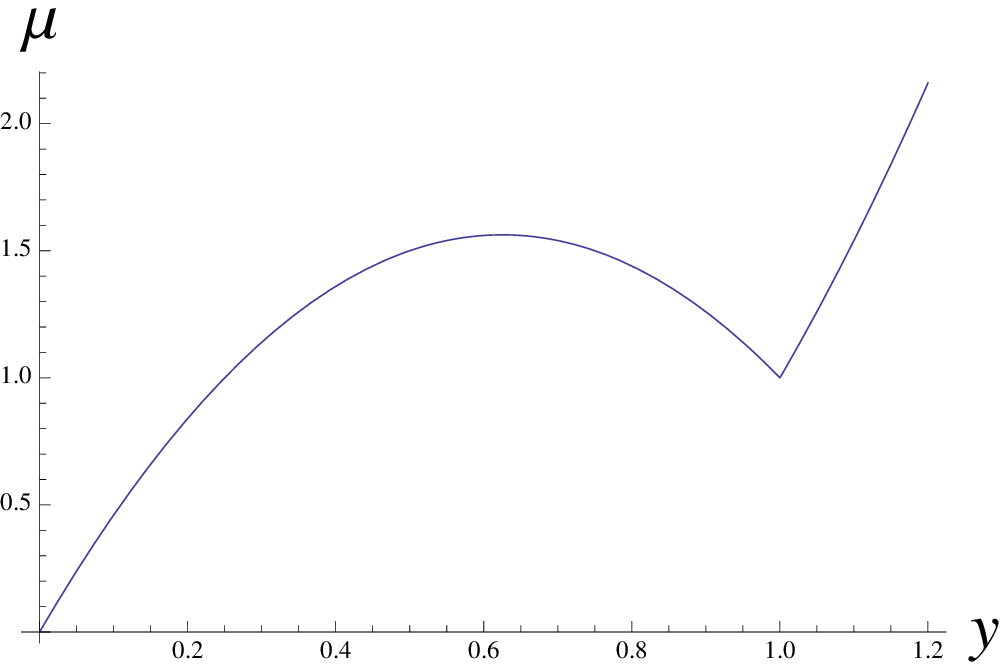}
		\caption{$ K > 1$}
		 \label{fig:bigA}
        \end{subfigure}
        \caption{Equilibria of \eqref{red1} for different values of $\mu$ in the cases (a) $K < 1$ and (b) $K >1$. }
        \label{fig:As}
\end{figure}
 
 If $K<1$, then $d\mu/dy>0$ for all $y$ except $y=1$ where the derivative is discontinuous.  Thus, for $K < 1 $ the curve of equilibria $\mu = \mu(y)$ is monotone increasing.  The system \eqref{red1} has a unique equilibrium solution.  The equilibrium is globally attracting, and it is important to remember that the solution corresponds to a unique stable circulation state (i.e., direction and strength).  However, if $K >1$ the system exhibits bistability for a range of $\mu$ values.  While $\mu(y)$ is still monotone increasing for $y > 1$, the curve has a local maximum at $ y =(1+K)/(2K) < 1$.  Thus for 
 $$\displaystyle{1 < \mu < \frac{(1+K)^2}{4K} } $$ 
 there are three equilibria.  The system is bistable with the outer two equilibria being stable, and the middle equilibrium being unstable.  Figure \ref{bifurcation} shows the bifurcation diagram for $K>1$.
 
	\begin{figure}[t]
		\begin{center}
			\includegraphics[width=0.4\textwidth]{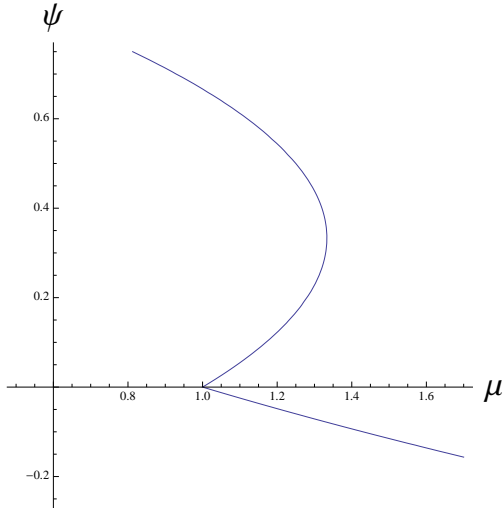}
		\caption{Bifurcation diagram for \eqref{red1}.  $\psi = 1-y.$}
		\label{bifurcation}
		\end{center}
	\end{figure}

It is hypothesized that the hysteresis loop implicitly appearing in Figure \ref{bifurcation} is important to understanding variability in paleoclimate data, specifically with regard to rapid warming events in the North Atlantic \cite{paleoclimate, dijkstraghil, thsaha}.  The bifurcation diagram suggests $\mu$ is the key to generating oscillatory behavior (when $K>1$), so we include $\mu$ as a dynamic variable.  We assume that $\mu$ evolves according to long-term average cloud formation, and that $y$ affects this process in a negative way.  Also, since $\mu$ was a parameter of the physical system, it is assumed to evolve on a longer time scale than $y$.  Thus it makes sense to consider the system
\begin{equation}
		\label{red2}
		\begin{array}{c}
			\dot{y} =  \mu - y - K | 1 - y | y \\
			\dot{\mu} =  \epsilon (\lambda - y),
		\end{array}
\end{equation}
where $K > 1$.  Here, we will look at this model as it relates to Theorem \ref{corner}.  For a more complete discussion of this model, see \cite{me2}.
	
The critical manifold of \eqref{red2} is precisely the set of equilibria of \eqref{red1} depicted in Figure \ref{fig:bigA}.  As $\lambda$ decreases through 1, the system undergoes a nonsmooth Hopf bifurcation as outlined by Theorem \ref{corner}.  That is, if $K < 1 +  2\sqrt{\epsilon}$ canard cycles appear at the corner.  However, if $K > 1 + 2\sqrt{\epsilon}$ the system undergoes a super-explosion.  In either case, the slope of the right stable branch is greater than 1, so the bifurcation is supercritical.

In the context of the ocean model \eqref{red2}, a stable equilibrium indicates a stable circulation state (direction and strength).  When $y < 1$ the ocean transports water poleward along the surface (sometimes called a `positive' circulation), and when $y > 1$ water is transported from the pole to the equator along the surface (called a `negative' circulation).  The bifurcation that occurs as $\lambda$ decreases through 1 indicates a change in the system from a stable negative circulation state to a periodically reversing circulation.  If canard cycles are created through the bifurcation, then, near the bifurcation point, the circulation will oscillate between weak circulation states in each direction.  Under a super-explosion, however, the circulation will oscillate between a weak negative (equatorward) circulation and a strong positive (poleward) circulation.  Thus, immediately after bifurcation, the relaxation oscillation created by the super-explosion is reminiscent of the desired hysteresis loop--a contrast to the headless canard cycles that will not reach the left stable branch of the critical manifold corresponding to a strong poleward circulation.

\section{Discussion}
In this paper, we have demonstrated the existence of canard cycles in planar nonsmooth fast/slow systems with a piecewise-smooth `2'-shaped critical manifold, pushing the theory beyond piecewise linear systems.  Through comparison with smooth shadow systems, we have shown that the amplitudes of canard cycles in nonsmooth systems are bounded by the amplitudes of canard cycles in corresponding smooth systems.  As we see in Figure \ref{fig:canSup}, it is possible for a corner to produce canards with head, and there is a delay between the bifurcation and the explosion phase.  This is a contrast to the piecewise-linear case, where the quasi-canards  are unable to produce the variety with heads, and the explosion phase begins immediately upon bifurcation.  In this respect, it is possible for a nonlinear, piecewise-smooth system to exhibit canards that are closer to their smooth cousins.  

On the other hand, the splitting line is essential for super-explosion. The instantaneous transition from a globally attracting equilibrium point to relaxation oscillations is a product of the nonsmooth nature of the system.  Whether the bifurcation produces canard cycles or causes an instantaneous jump to relaxation oscillations is determined by the slope of the repelling branch of the critical manifold at the splitting line; the slope of the attracting branch (relative to that of the repelling branch) only serves to determine the criticality of the bifurcation.  These slopes indicate the degree to which there is a local blending of time scales between the fast and slow variables.  Canard cycles require a local time-scale blending--at least on the side of splitting line containing the repelling branch of the critical manifold.  The possibility of a strong time-scale separation at a corner suggests that it is somehow easier for a nonsmooth system to exhibit relaxation oscillations, which is made apparent through super-explosions.

This paper does not put to bed entirely the theory of canards in piecewise-smooth planar systems.  We considered specifically the case where the splitting line and slow nullcline were both vertical lines (i.e., orthogonal to the fast direction) ensuring that the 'Hopf-like' bifurcations occur precisely at a fold or corner.  If the slow nullcline were to depend on the slow variable $y$, the bifurcation point no longer occurs at a fold or corner, complicating the analysis.    

\section*{Acknowledgments}
This research was made possible by the support of the Mathematics and Climate Research Network through a SAVI supplement allowing the first author to visit the University of Manchester.  The first author also thanks the University of Manchester for its hospitality during that visit.  The research was supported by the NSF under grants DMS-0940363 and DMS-1239013.


\providecommand{\bysame}{\leavevmode\hbox to3em{\hrulefill}\thinspace}
\providecommand{\MR}{\relax\ifhmode\unskip\space\fi MR }
\providecommand{\MRhref}[2]{%
  \href{http://www.ams.org/mathscinet-getitem?mr=#1}{#2}
}
\providecommand{\href}[2]{#2}

\end{document}